\documentclass[oneside,english]{amsart}
\usepackage[T1]{fontenc}
\usepackage[latin9]{inputenc}
\usepackage{color}
\usepackage{babel}
\usepackage{amsthm}
\usepackage{amstext}
\usepackage{amssymb}
\usepackage{graphicx}
\usepackage{esint}
\usepackage[unicode=true,
 bookmarks=true,bookmarksnumbered=false,bookmarksopen=false,
 breaklinks=false,pdfborder={0 0 1},backref=false,colorlinks=false]
 {hyperref}
\hypersetup{pdftitle={Multi-frequency acousto-electromagnetic tomography},
 pdfauthor={Giovanni S. Alberti, Habib Ammari, and Kaixi Ruan}}

\makeatletter

\theoremstyle{plain}
\newtheorem{thm}{\protect\theoremname}
  \theoremstyle{plain}
  \newtheorem{lem}[thm]{\protect\lemmaname}
  \theoremstyle{plain}
  \newtheorem{prop}[thm]{\protect\propositionname}

\@ifundefined{showcaptionsetup}{}{%
 \PassOptionsToPackage{caption=false}{subfig}}
\usepackage{subfig}
\makeatother

  \providecommand{\lemmaname}{Lemma}
  \providecommand{\propositionname}{Proposition}
\providecommand{\theoremname}{Theorem}

\begin{document}
\global\long\def\R{\mathbb{R}}

\global\long\def\N{\mathbb{N}}

\global\long\def\C{\mathbb{C}}

\global\long\def\bo{\partial\Omega}

\global\long\def\phi{\varphi}

\global\long\def\div{{\rm div}}

\global\long\def\curl{{\rm curl}}

\global\long\def\supp{{\rm supp}}

\global\long\def\A{\mathcal{A}}

\global\long\def\dq{\rho}

\global\long\def\omin{K_{min}}

\global\long\def\omax{K_{max}}

\global\long\def\dpsi{\xi}

\title[Multi-frequency acousto-electromagnetic tomography]{Multi-frequency acousto-electromagnetic tomography}
\author{Giovanni S. Alberti}
\author{Habib Ammari}
\address{Department of Mathematics and Applications,  Ecole Normale Sup\'erieure, 45 Rue d'Ulm, 75005 Paris, France.} 
\email{giovanni.alberti@ens.fr}
\email{habib.ammari@ens.fr}
\author{Kaixi Ruan}
\address{Institute for Computational \&
Mathematical Engineering,  Stanford, California 94305-2215, USA.} 
\email{kaixi@stanford.edu}
\date{October 15, 2014}

\thanks{This work was supported by the ERC Advanced Grant Project
MULTIMOD--267184.}

\begin{abstract}
This paper focuses on the acousto-electromagnetic tomography, a recently
introduced hybrid imaging technique. In a previous work, the reconstruction
of the electric permittivity of the medium from internal data was
achieved under the Born approximation assumption. In this work, we
tackle the general problem by a Landweber iteration algorithm. The
convergence of such scheme is guaranteed with the use of a multiple
frequency approach, that ensures uniqueness and stability for the
corresponding linearized inverse problem. Numerical simulations are
presented.
\end{abstract}
\maketitle

\section{Introduction}

In hybrid imaging inverse problems, two different techniques are combined
to obtain high resolution and high contrast images. More precisely,
two types of waves are coupled simultaneously: one gives high resolution,
and the other one high contrast. Much research has been done in the
last decade to develop and study several new methods; the reader is
referred to \cite{alberti-capdeboscq-2015, ammari2011expansion, ammari2008introduction, bal2012_review, kuchment2011_review}
for a review on hybrid techniques. A typical combination is between
ultrasonic waves and a high contrast wave, such as light or microwaves.
The high resolution of ultrasounds can be used to perturb the medium,
thereby changing the electromagnetic properties, and cross-correlating
electromagnetic boundary measurements lead to internal data (see e.g.
\cite{cap2008,cap2011,2014-ammari-bossy-garnier-nguyen-seppecher, ammari:hal-00778971, bal-moscow-2014,bal-schotland-2014}).

This paper focuses on the technique introduced in \cite{ammari:hal-00778971},
the so called acousto-electromagnetic tomography. Spherical ultrasonic
waves are sent from sources around the domain under investigation.
The pressure variations create a displacement in the tissue, thereby
modifying the electrical properties. Microwave boundary measurements
are taken in the unperturbed and in the perturbed situation (see Figure~\ref{experience1}).
In a first step, the cross-correlation of all the boundary values,
after the inversion of a spherical mean Radon transform, gives the internal
data of the form
\[
|u_{\omega}(x)|^{2}\nabla q(x),
\]
where $q$ is the spatially varying electric permittivity of the body
$\Omega\subset\R^{d}$ for $d=2,3$, $\omega>0$ is the frequency
and $u_{\omega}$ satisfies the Helmholtz equation with Robin boundary conditions
\begin{equation}\label{eq:helmholtz}
 \left\{ \begin{array}{l}
\Delta u_{\omega}+\omega^{2}qu_{\omega}=0\text{ in }\Omega,\\
\frac{\partial u_{\omega}}{\partial\nu}-i\omega u_{\omega}=-i\omega\phi\text{ on }\partial\Omega.
\end{array}\right.
\end{equation}
(In fact, only the gradient part $\psi_\omega$ of $|u_{\omega}|^{2}\nabla q = \nabla\psi_\omega +\curl\Phi$ is measured.) The second step of this hybrid methodology consists
in recovering $q$ from the knowledge of $\psi_{\omega}$. In \cite{ammari:hal-00778971}
an algorithm based on the inverse Radon transform was considered,
but it works only under the Born approximation, namely under the assumption
that $q$ has small variations around a certain constant value $q_{0}$.

The purpose of this work is to discuss a reconstruction algorithm
valid for general values of $q$. Denoting the measured datum by $\psi_{\omega}^{*}$,
we propose to minimize the energy functional
\[
J_{\omega}(q)=\frac{1}{2}\int_{\Omega}|\psi_{\omega}(q)-\psi_{\omega}^{*}|^{2}dx
\]
with a gradient descent method. In this case, this is equivalent to
a Landweber iteration scheme. The convergence of such algorithm \cite{1995-hanke-neubauer-scherzer}
is guaranteed provided that $\left\Vert D\psi_{\omega}[q](\dq)\right\Vert \ge C\left\Vert \dq\right\Vert $.
This condition represents the uniqueness and stability for the linearized
inverse problem $D\psi_{\omega}[q](\dq)\mapsto\dq$. This problem
has been studied for certain classes of internal functionals in \cite{kuchment2012stabilizing}
by looking at the ellipticity of the associated pseudo-differential
operator. Using these techniques, stability up to a finite dimensional
kernel could be established. However, uniqueness is a harder issue
\cite{kuchment2014stabilizing}, and in general only generic injectivity
can be proved. Indeed, the kernel of $\dq\mapsto D\psi_{\omega}[q](\dq)$
may well be non-trivial.

In order to obtain an injective problem, we propose here to use a
multiple frequency approach. If the boundary condition $\phi$ is
suitably chosen (e.g. $\phi=1$), the kernels of the operators $\dq\mapsto D\psi_{\omega}[q](\dq)$
``move'' as $\omega$ changes, and by choosing a finite number of
frequencies $K$ in a fixed range, determined a priori, it is possible
to show that the intersection becomes empty. In particular, there
holds $\sum_{\omega\in K}\left\Vert D\psi_{\omega}[q](\dq)\right\Vert \ge C\left\Vert \dq\right\Vert $
and the convergence of an optimal control algorithm for the functional $J=\sum_{\omega\in K}J_{\omega}$
follows \cite{2014arXiv1403.5708A} (see Theorem~\ref{thm:main}). 

The reader is referred to \cite{2014-ammari-waters-zhang, bal-2014-contemporary,montalto-stefanov-2013, widlak-scherzer-2014}
and references therein for recent works on uniqueness and stability
results on inverse problems from internal data. The use of multiple
frequencies to enforce non-zero constraints in PDE, and to obtain
well-posedness for several hybrid problems, has been discussed in
\cite{2013InvPr..29k5007A,2013-albertigs,2014-albertigs,2014albertidphil,alberti-capdeboscq-2014,2014arXiv1403.5708A}.

This paper is structured as follows. Section~\ref{sec:Acousto-Electromagnetic-Tomograp}
describes the physical model and the proposed optimization approach.
In Section~\ref{sec:Convergence-of-the} we prove the convergence
of the multi-frequency Landweber scheme. Some numerical simulations
are discussed in Section~\ref{sec:Numerical-Results}. Finally, Section~\ref{sec:Conclusions}
is devoted to some concluding remarks.

\section{\label{sec:Acousto-Electromagnetic-Tomograp}Acousto-Electromagnetic
Tomography}

In this section we recall the coupled physics inverse problem introduced
in \cite{ammari:hal-00778971} and discuss the proposed Landweber scheme.

\subsection{\label{sub:Physical-model}Physical Model}

\begin{figure}
\centering \includegraphics[scale=0.35]{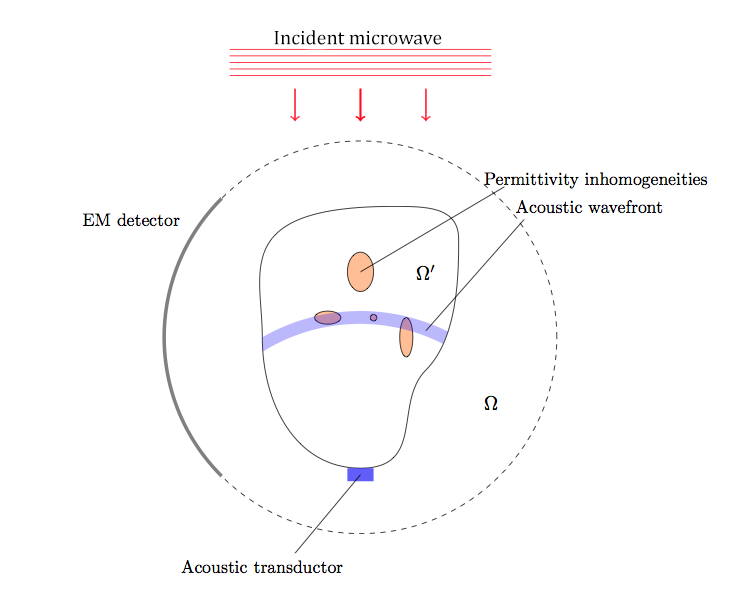} \caption{The acousto-electromagnetic tomography experiment.}

\label{experience1} 
\end{figure}

We now briefly describe how to measure the internal data in the hybrid
problem under consideration. The reader is referred to \cite{ammari:hal-00778971}
for full details.

Let $\Omega\subset\R^{d}$ be a bounded and smooth domain, for $d=2$
or $d=3$ and $q\in L^{\infty}(\Omega;\R)\cap H^{1}(\Omega;\R)$ be
the electric permittivity of the medium. We assume that $q$ is known
and constant near the boundary $\bo$, namely $q=1$ in $\Omega\setminus\Omega'$,
for some $\Omega'\Subset\Omega$. More precisely, suppose that $q\in Q$,
where for some $\Lambda>0$
\begin{equation}
Q:=\{q\in H^{1}(\Omega;\R):\Lambda^{-1}\le q\le\Lambda\text{ in \ensuremath{\Omega}, \ensuremath{\left\Vert q\right\Vert _{H^{1}(\Omega)}\le\Lambda}\ and \ensuremath{q=1\;}in \ensuremath{\Omega\setminus\Omega'}}\}.\label{eq:Q}
\end{equation}
In this paper, we model electromagnetic propagation in $\Omega$ at frequency
$\omega\in\A=[\omin,\omax]\subset\R_{+}$  by \eqref{eq:helmholtz}. The boundary value problem model allows us to consider arbitrary $q$
beyond the Born approximation, and so it is used here instead of the
free propagation model, which was originally considered in \cite{ammari:hal-00778971}.
Problem (\ref{eq:helmholtz}) admits a unique solution $u_{\omega}\in H^{1}(\Omega;\C)$
for a fixed boundary condition $\phi\in H^{1}(\Omega;\C)$ (see Lemma~\ref{lem:helmholtz full}).

Let us discuss how microwaves are combined with acoustic waves. A
short acoustic wave creates a displacement field $v$ in $\Omega$
(whose support is the blue area in Figure~\ref{experience1}), which
we suppose continuous and bijective. Then, the permittivity distribution
$q$ becomes $q_{v}$ defined by 
\[
q_{v}(x+v(x))=q(x),\qquad x\in\Omega,
\]
and the complex amplitude $u_{\omega}^{v}$ of the electric wave in
the perturbed medium satisfies 
\begin{equation}
\Delta u_{\omega}^{v}+\omega^{2}q_{v}u_{\omega}^{v}=0\text{ in }\Omega.\label{helmperturbed}
\end{equation}
Using (\ref{eq:helmholtz}) and (\ref{helmperturbed}), for $v$ small
enough we obtain the cross-correlation formula 
\[
\int_{\bo}(\frac{\partial u_{\omega}}{\partial n}\overline{u_{\omega}^{v}}-\frac{\partial\overline{v}}{\partial n}u_{\omega})\, d\sigma=\omega^{2}\int_{\Omega}(q_{v}-q)u_{\omega}\overline{u_{\omega}^{v}}\, dx\approx\omega^{2}\int_{\Omega}|u_{\omega}|^{2}\nabla q\cdot v\, dx,
\]
By boundary measurements, the left hand side of this equality is known.
Thus, we have measurements of the form
\[
\int_{\Omega}|u_{\omega}|^{2}\nabla q\cdot v\, dx,
\]
for all perturbations $v$. It is shown in \cite{ammari:hal-00778971,2014-ammari-bossy-garnier-nguyen-seppecher}
that choosing radial displacements $v$ allows to recover the gradient
part of $|u_{\omega}|^{2}\nabla q$ by using the inversion for the
spherical mean Radon transform. Namely, writing the Helmholtz decomposition
of $|u_{\omega}|^{2}\nabla q$
\[
|u_{\omega}|^{2}\nabla q=\nabla\psi_{\omega}+\curl\Phi_{\omega},
\]
for $\psi_{\omega}\in H^{1}(\Omega;\R)$ and $\Phi_{\omega}\in H^{1}(\Omega;\R^{2d-3})$,
the potential $\psi_{\omega}$ can be measured. Moreover, $\psi_{\omega}$
is the unique solution to \cite[Chapter I, Theorem 3.2 and Corollary 3.4]{GIRAULT-RAVIART-1986}
\begin{equation}
\left\{ \begin{array}{l}
\Delta\psi_{\omega}=\div(|u_{\omega}|^{2}\nabla q)\qquad\text{in }\Omega,\\
\frac{\partial\psi_{\omega}}{\partial\nu}=0\qquad\text{on }\bo,\\
\int_{\Omega}\psi_{\omega}\, dx=0.
\end{array}\right.\label{eq:psi_omega}
\end{equation}

In this paper, we assume that the inversion of the spherical mean Radon transform
has been performed and that we have access to $\psi_{\omega}$. In
the following, we shall deal with the second step of this hybrid imaging
problem: recovering the map $q$ from the knowledge of $\psi_{\omega}$.

\subsection{The Landweber Iteration}\label{sub:landweber}

Let $q^{*}$ be the real permittivity with corresponding measurements
$\psi_{\omega}^{*}$. Let $K\subset\A$ be a finite set of admissible
frequencies for which we have the measurements $\psi_{\omega}^{*}$,
$\omega\in K$. The set $K$ will be determined later. Let us denote
the error map by
\begin{equation}
F_{\omega}\colon Q\to H_{\nu}^{1}(\Omega;\R),\qquad q\mapsto\psi_{\omega}(q)-\psi_{\omega}^{*},\label{eq:map F}
\end{equation}
where $\psi_{\omega}(q)$ is the unique solution to (\ref{eq:psi_omega}),
and $H_{\nu}^{1}(\Omega;\R)=\{u\in H^{1}(\Omega;\R):\frac{\partial u}{\partial\nu}=0\text{ on \ensuremath{\bo}}\}$.

A natural approach to recover the real conductivity is to minimize
the discrepancy functional $J$ defined as 
\begin{equation}
J(q)=\frac{1}{2}\sum_{\omega\in K}\int_{\Omega}|F_{\omega}(q)|^{2}dx,\qquad q\in Q.\label{J}
\end{equation}
The gradient descent method can be employed to minimize $J$. At each
iteration we compute 
\[
q_{n+1}=T\left(q_{n}-hDJ[q_{n}]\right),
\]
where $h>0$ is the step size and $T\colon H^{1}(\Omega;\R)\to Q$
is the Hilbert projection onto the convex closed set $Q$, which guarantees
that at each iteration $q_{n}$ belongs to the admissible set $Q$.
Since $DJ[q]=\sum_{\omega}DF_{\omega}[q]^{*}(F_{\omega}(q))$, this
algorithm is equivalent to the Landweber scheme \cite{1995-hanke-neubauer-scherzer}
given by 
\begin{equation}
q_{n+1}=T\Bigl(q_{n}-h\sum_{\omega\in K}DF_{\omega}(q_{n})^{*}(F_{\omega}(q_{n}))\Bigr).\label{algoLandweber}
\end{equation}
(For the Fr\'echet differentiability of the map $F_{\omega}$, see Lemma~\ref{lem:frechet differentiability}.)

The main result of this paper states that the Landweber scheme defined
above converges to the real unknown $q^{*}$, provided that $K$ is
suitably chosen and that $h$ and $\left\Vert q_{0}-q^{*}\right\Vert _{H^{1}(\Omega)}$
are small enough. The most natural choice for the set of frequencies
$K$ is as a uniform sample of $\A$, namely let
\[
K^{(m)}=\{\omega_{1}^{(m)},\dots,\omega_{m}^{(m)}\},\qquad\omega_{i}^{(m)}=\omin+\frac{(i-1)}{(m-1)}(\omax-\omin).
\]

\begin{thm}
\label{thm:main}Set $\phi=1$. There exist $C>0$ and $m\in\N^{*}$
depending only on $\Omega$, $\Lambda$ and $\A$ such that for any
$q\in Q$ and $\dq\in H_{\nu}^{1}(\Omega;\R)$ 
\begin{equation}
\sum_{\omega\in K^{(m)}}\left\Vert DF_{\omega}[q](\dq)\right\Vert {}_{H^{1}(\Omega;\R)}d\omega\ge C\left\Vert \dq\right\Vert {}_{H^{1}(\Omega;\R)}.\label{eq:Landweberinequality}
\end{equation}
As a consequence, the sequence defined in (\ref{algoLandweber}) converges
to $q^{*}$ provided that $h$ and $\left\Vert q_{0}-q^{*}\right\Vert _{H^{1}(\Omega)}$
are small enough.
\end{thm}
The proof of this theorem is presented in Section~\ref{sec:Convergence-of-the}.
In view of the results in \cite{1995-hanke-neubauer-scherzer,2014arXiv1403.5708A},
the convergence of the Landweber iteration follows from the Lipschitz
continuity of $F_{\omega}$ and from inequality (\ref{eq:Landweberinequality}).
The Lipschitz continuity of $F_{\omega}$ is a simple consequence
of the elliptic theory. 

On the other hand, the lower bound given in (\ref{eq:Landweberinequality})
is non-trivial, since it represents the uniqueness and stability of
the multi-frequency linearized inverse problem
\[
(DF_{\omega}[q](\dq))_{\omega\in K^{(m)}}\longmapsto\dq.
\]
As it has been discussed in the Introduction, the kernels of the operators
$\dq\mapsto DF_{\omega}[q](\dq)$ ``move'' as $\omega$ changes.
More precisely, the intersection of the kernels corresponding to the
a priori determined finite set of frequencies $K^{(m)}$ is empty. Moreover,
the argument automatically gives an a priori constant $C$ in (\ref{eq:Landweberinequality}).

The multi-frequency method is based on the analytic dependence of
the problem with respect to the frequency $\omega$, and on the fact
that in $\omega=0$ the problem is well posed. Indeed, when $\omega\to0$
it is easy to see that $ $$u_{\omega}\to1$ in (\ref{eq:helmholtz}),
so that $u_{0}=1$. Thus, looking at (\ref{eq:psi_omega}), the measurement
datum $\psi_{0}$ is nothing else than $q^{*}$ (up to a constant).
Therefore, $q^{*}$ could be easily determined when $\omega=0$ since
$q^{*}$ is known on the boundary $\bo$. As we show in the following
section, the analyticity of the problem with respect to $\omega$
allows to ``transfer'' this property to the desired range of frequencies
$\A$.

\section{\label{sec:Convergence-of-the}Convergence of the Landweber Iteration}

In order to use the well-posedness of the problem in $\omega=0$ we
shall need the following result on quantitative unique continuation
for vector-valued holomorphic functions.
\begin{lem}
\label{lem:unique continuation}Let $V$ be a complex Banach space,
$\A=[\omin,\omax]\subset\R_{+}$, $C_{0},D>0$ and $g\colon B(0,\omax)\to V$
be holomorphic such that $\left\Vert g(0)\right\Vert \ge C_{0}$ and
\[
\sup_{\omega\in B(0,\omax)}\left\Vert g(\omega)\right\Vert \le D.
\]
Then there exists $\omega\in\A$ such that
\[
\left\Vert g(\omega)\right\Vert \ge C
\]
for some $C>0$ depending only on $\A$, $C_{0}$ and $D$.\end{lem}
\begin{proof}
By contradiction, assume that there exists a sequence of holomorphic
functions $g_{n}\colon B(0,\omax)\to V$ such that $\left\Vert g_{n}(0)\right\Vert \ge C_{0}$,
$\sup_{\omega\in B(0,\omax)}\left\Vert g_{n}(\omega)\right\Vert \le D$
and $\max_{\omega\in\A}\left\Vert g_{n}(\omega)\right\Vert \to0$.
By Hahn Banach theorem, for any $n$ there exists $T_{n}\in V'$ such
that $\left\Vert T_{n}\right\Vert \le1$ and $T_{n}(g_{n}(0))=\left\Vert g_{n}(0)\right\Vert $.
Set $f_{n}:=T_{n}\circ g_{n}\colon B(0,\omax)\to\C$. Thus $(f_{n})$
is a family of complex-valued uniformly bounded holomorphic functions,
since
\[
\left|f_{n}(\omega)\right|\le\left\Vert T_{n}\right\Vert \left\Vert g_{n}(\omega)\right\Vert \le D,\qquad\omega\in B(0,\omax).
\]
As a consequence, by standard complex analysis, there exists a holomorphic
function $f\colon B(0,\omax)\to\C$ such that $f_{n}\to f$ uniformly.
We readily observe that for any $\omega\in\A$ there holds
\[
\left|f(\omega)\right|=\lim_{n}\left|f_{n}(\omega)\right|\le\lim_{n}\left\Vert T_{n}\right\Vert \left\Vert g_{n}(\omega)\right\Vert =0,
\]
since $\max_{\omega\in\A}\left\Vert g_{n}(\omega)\right\Vert \to0$.
By the unique continuation theorem $f(0)=0$. On the other hand, as
$T_{n}(g_{n}(0))=\left\Vert g_{n}(0)\right\Vert $, 
\[
f(0)=\lim_{n}f_{n}(0)=\lim\left\Vert g_{n}(0)\right\Vert \ge C_{0}>0,
\]
which yields a contradiction.
\end{proof}
In view of (\ref{eq:Landweberinequality}), we need to
study the Fr\'echet differentiability of the map $F_{\omega}$ and characterize
its derivative. Before doing this, we study the well-posedness of
(\ref{eq:helmholtz}). The result is classical; for a proof, see \cite[Section 8.1]{melenk1995generalized}.
\begin{lem}
\label{lem:helmholtz full}Let $\Omega\subset\R^{d}$ be a bounded
and smooth domain for $d=2,3$, $\omega\in B(0,\omax)$ and $q\in Q$.
For any $f\in L^{2}(\Omega;\C)$ and $\phi\in H^{1}(\Omega;\C)$ the
problem
\begin{equation}
\left\{ \begin{array}{l}
\Delta u+\omega^{2}qu=\omega f\qquad\text{in }\Omega,\\
\frac{\partial u}{\partial\nu}-i\omega u=-i\omega\phi\qquad\text{on }\partial\Omega,
\end{array}\right.\label{eq:helmholtz full}
\end{equation}
augmented with the condition
\begin{equation}
\int_{\bo}u\, d\sigma=\int_{\bo}\phi\, d\sigma-i\int_{\Omega}f\, dx\label{eq:mistery}
\end{equation}
if $\omega=0$ admits a unique solution $u\in H^{2}(\Omega;\C)$.
Moreover
\[
\left\Vert u\right\Vert _{H^{2}(\Omega;\C)}\le C\bigl(\left\Vert f\right\Vert _{L^{2}(\Omega;\C)}+\left\Vert \phi\right\Vert _{H^{1}(\Omega;\C)}\bigr)
\]
for some $C>0$ depending only on $\Omega$, $\Lambda$ and $\omax$.
\end{lem}
Since for $\omega=0$ the solution to \eqref{eq:helmholtz full} is unique up to a constant, condition
(\ref{eq:mistery}) is needed to have uniqueness. Even though it may
seem mysterious, this condition is natural in order to ensure continuity
of $u$ with respect to $\omega$. Indeed an integration by parts
gives
\[
\omega\int_{\Omega}f\, dx=\int_{\bo}\frac{\partial u}{\partial\nu}\, d\sigma+\omega^{2}\int_{\Omega}qu\, dx=i\omega\int_{\bo}(u-\phi)\, d\sigma+\omega^{2}\int_{\Omega}qu\, dx,
\]
whence for $\omega\neq0$ we obtain
\[
\int_{\bo}u\, d\sigma=\int_{\bo}\phi\, d\sigma-i\int_{\Omega}f\, dx+\omega i\int_{\Omega}qu\, dx,
\]
and so for $\omega=0$ we are left with (\ref{eq:mistery}). The
above condition is a consequence of (\ref{eq:helmholtz full}) for
$\omega\neq0$, but needs to be added in the case $\omega=0$ to guarantee
uniqueness.

Let us go back to (\ref{eq:helmholtz}). Fix $\phi\in H^{1}(\Omega;\C)$
and $\omega\in B(0,\omax)$. By Lemma~\ref{lem:helmholtz full} the
problem
\begin{equation}
\left\{ \begin{array}{l}
\Delta u_{\omega}+\omega^{2}qu_{\omega}=0\text{ in }\Omega,\\
\frac{\partial u_{\omega}}{\partial n}-i\omega u_{\omega}=-i\omega\phi\text{ on }\partial\Omega,
\end{array}\right.\label{eq:helmholtz 2}
\end{equation}
together with condition
\begin{equation}
\int_{\bo}u_{\omega}\, d\sigma=\int_{\bo}\phi\, d\sigma+\omega i\int_{\Omega}qu_{\omega}\, dx\label{eq:mistery 2}
\end{equation}
admits a unique solution $u_{\omega}\in H^{2}(\Omega;\C)$ such that
\begin{equation}
\left\Vert u_{\omega}\right\Vert _{H^{2}(\Omega;\C)}\le C\left\Vert \phi\right\Vert _{H^{1}(\Omega;\C)}\label{eq:uomega H2}
\end{equation}
for some $C>0$ depending only on $\Omega$, $\Lambda$ and $\omax$.
As above, (\ref{eq:mistery 2}) guarantees uniqueness and continuity
in $\omega=0$ and is implicit in (\ref{eq:helmholtz 2}) if $\omega\neq0$.

Next, we study the dependence of $u_{\omega}$ on $\omega$.
\begin{lem}
\label{lem:holomorphic}Let $\Omega\subset\R^{d}$ be a bounded
and smooth domain for $d=2,3$, $q\in Q$ and $\phi\in H^{1}(\Omega;\C)$.
The map
\[
\mathcal{F}(q)\colon B(0,\omax)\longrightarrow H^{2}(\Omega;\C),\qquad\omega\longmapsto u_{\omega}
\]
is holomorphic. Moreover, the derivative $\partial_{\omega}u_{\omega}\in H^{2}(\Omega;\C)$
is the unique solution to
\begin{equation}
\left\{ \begin{array}{l}
\Delta\partial_{\omega}u_{\omega}+\omega^{2}q\partial_{\omega}u_{\omega}=-2\omega qu_{\omega}\text{ in }\Omega,\\
\frac{\partial(\partial_{\omega}u_{\omega})}{\partial\nu}-i\omega\partial_{\omega}u_{\omega}=iu_{\omega}-i\phi\text{ on }\partial\Omega,
\end{array}\right.\label{eq:helmholtz derivative}
\end{equation}
together with condition
\begin{equation}
\int_{\bo}\partial_{\omega}u_{\omega}\, d\sigma=\omega i\int_{\Omega}q\partial_{\omega}u_{\omega}\, dx+i\int_{\Omega}qu_{\omega}\, dx,\label{eq:mistery derivative}
\end{equation}
and satisfies
\[
\left\Vert \partial_{\omega}u_{\omega}\right\Vert _{H^{2}(\Omega;\C)}\le C\left\Vert \phi\right\Vert _{H^{1}(\Omega;\C)}
\]
for some $C>0$ depending only on $\Omega$, $\Lambda$ and $\omax$.\end{lem}
\begin{proof}
The proof of this result is completely analogous to the ones given
in \cite{2013InvPr..29k5007A,2013-albertigs,2014arXiv1403.5708A}
in similar situations. Here only a sketch will be presented.

Fix $\omega\in B(0,\omax)$: we shall prove that $\mathcal{F}(q)$
is holomorphic in $\omega$ and that the derivative is $\partial_{\omega}u_{\omega}$,
i.e., the unique solution to (\ref{eq:helmholtz derivative})-(\ref{eq:mistery derivative}).
For $h\in\C$ let $v_{h}=(u_{\omega+h}-u_{\omega})/h$. We need to
prove that $v_{h}\to\partial_{\omega}u_{\omega}$ in $H^{2}(\Omega)$
as $h\to0$. Suppose first $\omega\neq0$. A direct calculation shows
that
\[
\left\{ \begin{array}{l}
\Delta v_{h}+\omega^{2}qv_{h}=-2\omega qu_{\omega+h}-hqu_{\omega+h}\text{ in }\Omega,\\
\frac{\partial v_{h}}{\partial\nu}-i\omega v_{h}=i(u_{\omega+h}-\phi)\text{ on }\partial\Omega.
\end{array}\right.
\]
Arguing as in Lemma~\ref{lem:helmholtz full}, we obtain $u_{\omega+h}\to u_{\omega}$
as $h\to0$ in $H^{2}(\Omega)$, whence $v_{h}\to\partial_{\omega}u_{\omega}$
in $H^{2}(\Omega)$, as desired.

When $\omega=0$, the above system must be augmented with the condition

\[
\int_{\bo}v_{h}\, d\sigma=i\int_{\Omega}qu_{0}\, dx,
\]
which is a simple consequence of (\ref{eq:mistery 2}), and the result
follows.
\end{proof}
We now study the Fr\'echet differentiability of the map $F_{\omega}$
defined in (\ref{eq:map F}). The proof of this result is trivial,
and the details are left to the reader.
\begin{lem}
\label{lem:frechet differentiability}Let $\Omega\subset\R^{d}$
be a bounded and smooth domain for $d=2,3$, $q\in Q$, $\omega\in B(0,\omax)$
and $\phi\in H^{1}(\Omega;\C)$. The map $F_{\omega}$ is Fr\'echet
differentiable and for $\dq\in H_{\nu}^{1}(\Omega;\R)$, the derivative
$\dpsi_{\omega}(\dq):=DF_{\omega}[q](\dq)$ is the unique solution
to the problem
\[
\left\{ \begin{array}{l}
\Delta\dpsi_{\omega}(\dq)=\div\left(|u_{\omega}|^{2}\nabla\dq+(\overline{u}_{\omega}v_{\omega}(\dq)+u_{\omega}\overline{v}_{\omega}(\dq))\nabla q\right)\text{ in }\Omega,\\
\frac{\partial\dpsi_{\omega}(\dq)}{\partial\nu}=0\text{ on }\partial\Omega,\\
\int_{\Omega}\dpsi_{\omega}(\dq)\, dx=0,
\end{array}\right.
\]
where $v_{\omega}(\dq)\in H^{2}(\Omega;\C)$ is the unique solution
to
\begin{equation}
\left\{ \begin{array}{l}
\Delta v_{\omega}(\dq)+\omega^{2}qv_{\omega}(\dq)=-\omega^{2}\dq\, u_{\omega}\text{ in }\Omega,\\
\frac{\partial v_{\omega}(\dq)}{\partial\nu}-i\omega v_{\omega}(\dq)=0\text{ on }\partial\Omega,
\end{array}\right.\label{eq:v omega}
\end{equation}
together with $\int_{\bo}v_{0}(\dq)\, d\sigma=0$ if $\omega=0$.
In particular, $F_{\omega}$ is Lipschitz continuous, namely
\[
\left\Vert \dpsi_{\omega}(\dq)\right\Vert _{H^{1}(\Omega;\R)}\le C(\Omega,\Lambda,\omax,\left\Vert \phi\right\Vert _{H^{1}(\Omega;\C)})\left\Vert \dq\right\Vert _{H^{1}(\Omega;\R)}.
\]

\end{lem}
The main step in the proof of Theorem~\ref{thm:main} is inequality
(\ref{eq:Landweberinequality}), which we now prove. The argument
in the proof clarifies the multi-frequency method illustrated in the
previous section. The proof is structured as the proof of \cite[Theorem 1]{2014-albertigs}.
\begin{prop}
\label{prop:inequality}Set $\phi=1$. There exist $C>0$ and $m\in\N^{*}$
depending on $\Omega$, $\Lambda$ and $\A$ such that for any $q\in Q$
and $\dq\in H_{\nu}^{1}(\Omega;\R)$ 
\[
\sum_{\omega\in K^{(m)}}\left\Vert DF_{\omega}[q](\dq)\right\Vert {}_{H^{1}(\Omega;\R)}d\omega\ge C\left\Vert \dq\right\Vert {}_{H^{1}(\Omega;\R)}.
\]
\end{prop}
\begin{proof}
In the proof, several positive constants depending only on $\Omega$,
$\Lambda$ and $\A$ will be denoted by $C$ or $Z$.

Fix $q\in Q$. For $\dq\in H_{\nu}^{1}(\Omega;\R)$ such that $\left\Vert \dq\right\Vert _{H^{1}(\Omega;\R)}=1$
define the map
\[
g_{\dq}(\omega)=\div\left(u_{\omega}\overline{u}_{\overline{\omega}}\nabla\dq+(\overline{u}_{\overline{\omega}}v_{\omega}(\dq)+u_{\omega}\overline{v}_{\overline{\omega}}(\dq))\nabla q\right),\qquad\omega\in B(0,\omax).
\]
Hence $g_{\dq}\colon B(0,\omax)\to H_{\nu}^{1}(\Omega;\C)'$ is holomorphic.
We shall apply Lemma~\ref{lem:unique continuation} to $g_{\dq}$,
and so we now verify the hypotheses.

Since $\phi=1$, by (\ref{eq:helmholtz 2})-(\ref{eq:mistery 2})
we have $u_{0}=1$ and by (\ref{eq:v omega}) we have $v_{0}(\dq)=0$,
whence $g_{\dq}(0)=\div(\nabla\dq)$. Since $\frac{\partial\dq}{\partial\nu}=0$
on $\bo$ there holds
\[
\left\Vert g_{\dq}(0)\right\Vert _{H_{\nu}^{1}(\Omega;\C)'}=\left\Vert \div(\nabla\dq)\right\Vert _{H_{\nu}^{1}(\Omega;\C)'}\ge C\left\Vert \nabla\dq\right\Vert _{L^{2}(\Omega)}\ge C>0,
\]
since $\left\Vert \dq\right\Vert _{H^{1}(\Omega;\R)}=1$. For $\omega\in B(0,\omax)$
we readily derive
\[
\begin{split}\left\Vert g_{\dq}(\omega)\right\Vert _{H_{\nu}^{1}(\Omega;\C)'} & \le C\left\Vert u_{\omega}\overline{u}_{\overline{\omega}}\nabla\dq+(\overline{u}_{\overline{\omega}}v_{\omega}(\dq)+u_{\omega}\overline{v}_{\overline{\omega}}(\dq))\nabla q\right\Vert _{L^{2}(\Omega)}\\
 & \le C\left(\left\Vert \dq\right\Vert _{H^{1}(\Omega)}+\left\Vert q\right\Vert _{H^{1}(\Omega)}\right)\\
 & \le C,
\end{split}
\]
where the second inequality follows from (\ref{eq:uomega H2}), Lemma~\ref{lem:helmholtz full}
applied to $v_{\omega}(\dq)$ and the Sobolev embedding $H^{2}\hookrightarrow L^{\infty}$.
Therefore, by Lemma~\ref{lem:unique continuation} there exists $\omega_{\dq}\in\A$
such that
\begin{equation}
\left\Vert g_{\dq}(\omega_{\dq})\right\Vert _{H_{\nu}^{1}(\Omega;\C)'}\ge C.\label{eq:first bound}
\end{equation}

Consider now for $\omega\in B(0,\omax)$
\[
g_{\dq}'(\omega)=\div\left((u'_{\omega}\overline{u}_{\overline{\omega}}+u{}_{\omega}\overline{u'}_{\overline{\omega}})\nabla\dq+(\overline{u'}_{\overline{\omega}}v_{\omega}(\dq)+\overline{u}_{\overline{\omega}}v'_{\omega}(\dq)+u'_{\omega}\overline{v}_{\overline{\omega}}(\dq)+u_{\omega}\overline{v'}_{\overline{\omega}}(\dq))\nabla q\right)
\]
where for simplicity the partial derivative $\partial_{\omega}$ is
replaced by $'$. Arguing as before, and using Lemma~\ref{lem:holomorphic}
we obtain
\[
\left\Vert g'_{\dq}(\omega)\right\Vert _{H_{\nu}^{1}(\Omega;\C)'}\le C,\qquad\omega\in B(0,\omax).
\]
As a consequence, by (\ref{eq:first bound}) we obtain 
\begin{equation}
\left\Vert g_{\dq}(\omega)\right\Vert _{H_{\nu}^{1}(\Omega;\C)'}\ge C,\qquad\omega\in[\omega_{\dq}-Z,\omega_{\dq}+Z]\cap\A.\label{eq:2nd bound}
\end{equation}
Since $\A=[\omin,\omax]$ 
there exists $P=P(Z,\A)\in\N$ such that
\begin{equation}
\A\subseteq\bigcup_{p=1}^{P}I_{p},\qquad I_{p}=[\omin+(p-1)Z,\omin+pZ].\label{eq:partition}
\end{equation}
Choose now $m\in\N^*$ big enough so that for every $p=1,\dots,P$ there
exists $i_{p}=1,\dots,m$ such that $\omega(p):=\omega_{i_{p}}^{(m)}\in I_{p}$ (recall that  $\omega_{i}^{(m)}=\omin+\frac{(i-1)}{(m-1)}(\omax-\omin)$).
Note that $m$ depends only on $Z$ and $\left|\A\right|$.

Since $\left|[\omega_{\dq}-Z,\omega_{\dq}+Z]\right|=2Z$ and $\left|I_{p}\right|=Z$,
in view of (\ref{eq:partition}) there exists $p_{\dq}=1,\dots,P$
such that $I_{p_{\rho}}\subseteq[\omega_{\dq}-Z,\omega_{\dq}+Z]$. Therefore
$\omega(p_{\dq})\in[\omega_{\dq}-Z,\omega_{\dq}+Z]\cap\A$, whence
by (\ref{eq:2nd bound}) there holds $\left\Vert g_{\dq}(\omega(p_{\dq}))\right\Vert _{H_{\nu}^{1}(\Omega;\C)'}\ge C$.
Since $\omega(p_{\dq})\in\R$ this implies 
\[
\left\Vert \div\left(|u_{\omega(p_{\dq})}|^{2}\nabla\dq+(\overline{u}_{\omega(p_{\dq})}v_{\omega}(\dq)+u_{\omega}\overline{v}_{\omega(p_{\dq})}(\dq))\nabla q\right)\right\Vert _{H_{\nu}^{1}(\Omega;\C)'}\ge C,
\]
which by Lemma~\ref{lem:frechet differentiability} yields $\left\Vert \Delta\dpsi_{\omega(p_{\dq})}(\dq)\right\Vert _{H_{\nu}^{1}(\Omega;\C)'}\ge C$.
Hence, since $\frac{\partial\dpsi_{\omega(p_{\dq})}(\dq)}{\partial\nu}=0$
on $\bo$, there holds $\left\Vert \dpsi_{\omega(p_{\dq})}(\dq)\right\Vert _{H^{1}(\Omega)}\ge C$.
Thus, since $\omega(p_{\dq})\in K^{(m)}$
\[
\sum_{\omega\in K^{(m)}}\left\Vert \dpsi_{\omega}(\dq)\right\Vert _{H^{1}(\Omega)}\ge C.
\]
We have proved this inequality only for $\dq\in H_{\nu}^{1}(\Omega;\R)$
with unitary norm. By using the linearity of $\dpsi_{\omega}(\dq)$
with respect to $\dq$ we immediately obtain
\[
\sum_{\omega\in K^{(m)}}\left\Vert \dpsi_{\omega}(\dq)\right\Vert _{H^{1}(\Omega)}\ge C\left\Vert \dq\right\Vert _{H^{1}(\Omega)},\qquad\dq\in H_{\nu}^{1}(\Omega;\R),
\]
as desired.
\end{proof}
We are now ready to prove Theorem~\ref{thm:main}.
\begin{proof}[Proof of Theorem~\ref{thm:main}]
Inequality (\ref{eq:Landweberinequality}) follows from Proposition~\ref{prop:inequality}.
Moreover, $F_{\omega}$ is Lipschitz continuous by Lemma~\ref{lem:frechet differentiability}.
Therefore, the convergence of the Landweber iteration is a consequence
of the results in \cite{2014arXiv1403.5708A, 1995-hanke-neubauer-scherzer},
provided that $\left\Vert q_{0}-q^{*}\right\Vert _{H^{1}(\Omega)}$
and $h$ are small enough.
\end{proof}

\section{\label{sec:Numerical-Results}Numerical Results}

In this section we present some numerical results. Let $\Omega$ be the unit square $[0,1]\times[0,1]$. We set the mesh size to be $0.01$. A phantom image is used for the true permittivity distribution $q^*$ (see Figure~\ref{fig:sub1}). According to Theorem~\ref{thm:main} we set the Robin boundary condition to be a constant function $\phi=1$. Let $K$ be the set of frequencies for which we have measurements $\psi_\omega^*$, $\omega\in K$. As discussed in $\S$~\ref{sub:landweber}, we minimize the functional $J$ in \eqref{J} with the Landweber iteration scheme given in \eqref{algoLandweber}. The initial guess is $q_0=1$.

\begin{figure}[b]
\includegraphics[width=0.5\linewidth]{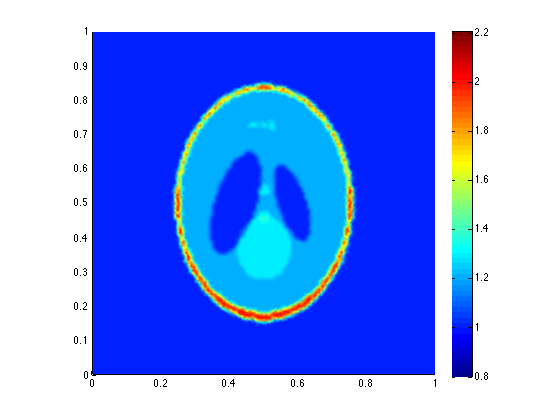}
\caption{The true permittivity distribution $q^*$.}
\label{fig:sub1}
\end{figure}

We start with the  imaging problem at a single frequency. In Figure~\ref{Q2} we display the findings for the case $K=\{3\}$. Figure~\ref{fig:sub2} shows the reconstructed distribution  after 100 iterations. Figure~\ref{erreurq} shows the relative error  as a function of the number of iterations. This suggests the convergence of the iterative algorithm, even though the algorithm is proved to be convergent only in the multi-frequency case. It is possible that for small frequencies $\omega$  (with respect to the domain size) we are still in the coercive case, i.e. the kernel  $R_\omega$ of $\rho\mapsto DF_\omega[q](\rho)$ is trivial and a single frequency is sufficient.

\begin{figure}

\subfloat[Reconstructed distribution after 100 iterations.]{
\includegraphics[width=0.48\linewidth]{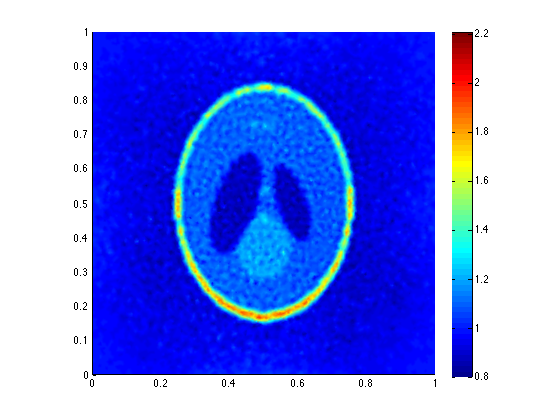}
\label{fig:sub2}
}
\subfloat[Relative error depending on the number of iterations.]{
\includegraphics[width=0.48\linewidth]{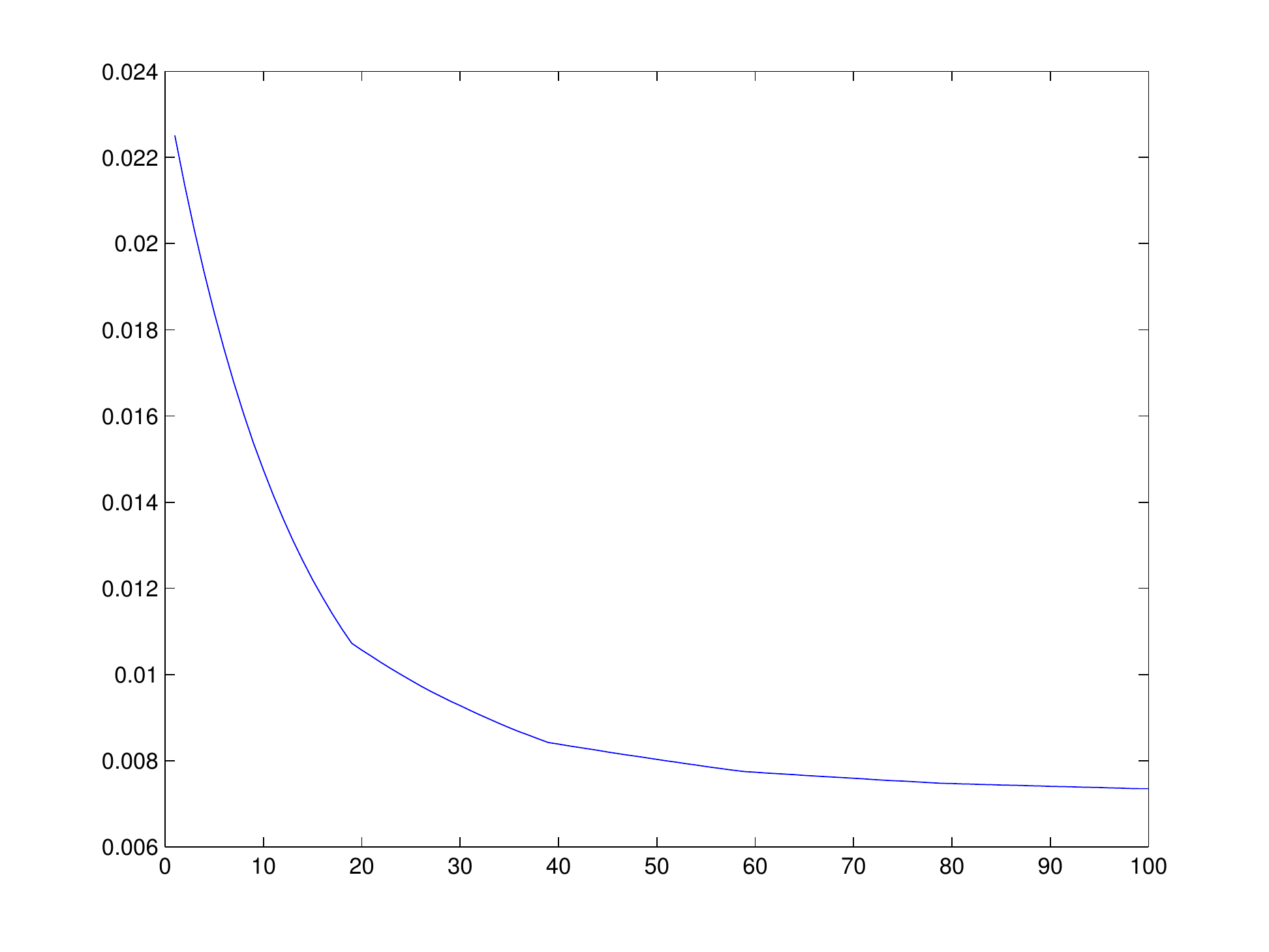}
\label{erreurq}
}
\caption{Reconstruction of $q$ for the set of frequencies $K=\{3\}$.}
\label{Q2} 
\end{figure}

However, this does not work at higher frequencies (with respect to the domain size). Figure~\ref{omegabig} shows some reconstructed maps for $\omega=10$, $\omega=15$ and $\omega=20$, which suggest that the algorithm may not converge numerically for high frequencies. In each case, there are areas that remain invisible.  This may be an indication that $R_\omega\neq\{0\}$  for these values of the frequency. 

\begin{figure}
\subfloat[$K=\{10\}$]{
\includegraphics[width=0.48\linewidth]{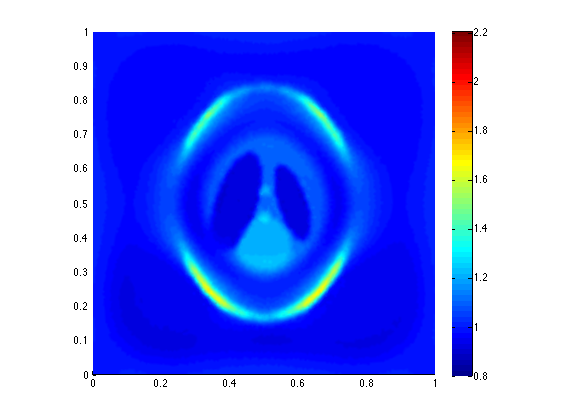}
\label{om10}
}
\subfloat[$K=\{15\}$]{
\includegraphics[width=0.48\linewidth]{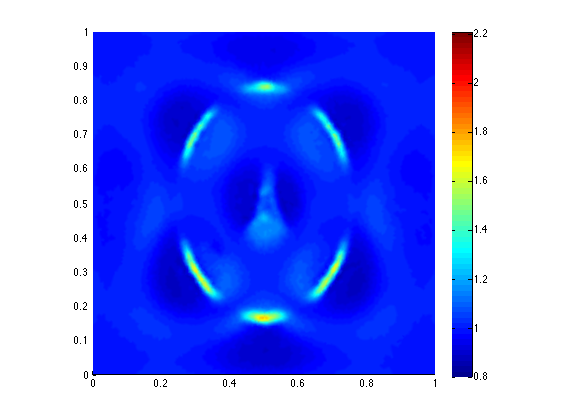}
\label{om15}
}

\subfloat[$K=\{20\}$]{
\includegraphics[width=0.48\linewidth]{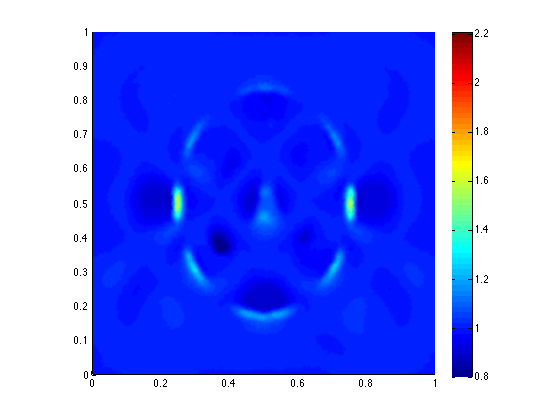}
\label{om20}
}

\caption{Reconstruction of $q$ for  higher frequencies.}
\label{omegabig} 
\end{figure}

The invisible areas in Figure~\ref{omegabig} are different for different frequencies, and so combining these measurements may give a satisfactory reconstruction. More precisely, according to Theorem~\ref{thm:main}, by using multiple frequencies it is possible to make the problem injective, namely $\cap_\omega R_\omega=\{0\}$, since the kernels $R_\omega$ change as $\omega$ varies. Figure~\ref{bigomegas} shows the results for the case $K=\left\{10,15,20\right\}$. (According to the notation introduced in Section~\ref{sec:Acousto-Electromagnetic-Tomograp}, this choice of frequencies corresponds to $\A=[10,20]$ and $m=3$.) These  findings suggest the convergence of the multi-frequency Landweber iteration, even though it was not convergent in each single-frequency case. Since we chose higher frequencies, the convergence is slower.

\begin{figure}
\subfloat[Reconstructed distribution after 200 iterations.]{

\includegraphics[width=0.48\linewidth]{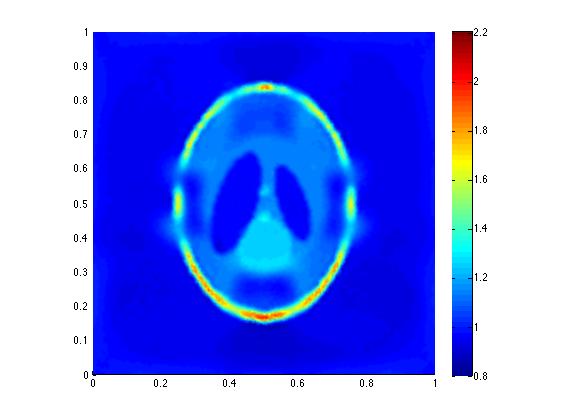}\label{101520}
} \subfloat[Relative error depending on the number of iterations.]{

\includegraphics[width=0.48\linewidth]{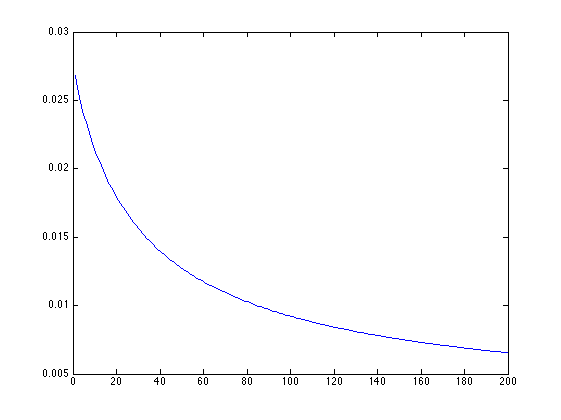}\label{er_b}

}
\caption{Reconstruction of $q$ for $K=\{10,15,20\}$.}
\label{bigomegas} 
\end{figure}

\section{\label{sec:Conclusions}Concluding Remarks}

In this paper, we proved that  the Landweber scheme in acousto-electromagnetic tomography converges to the true solution provided that multi-frequency measurements are used. We illustrated this result with several numerical examples. It would be challenging to estimate the robustness of the proposed algorithm with respect to random fluctuations in the electromagnetic parameters. This will be the subject of a forthcoming work.

\bibliographystyle{plainurl}
\bibliography{2014-alberti-ammari-ruan}

\end{document}